\documentclass[12pt,reqno]{amsart}
\usepackage[utf8]{inputenc}
\usepackage{amssymb, amsmath, amsthm, url, mathrsfs}
\usepackage{hyperref}
\usepackage{todonotes}
\usepackage{enumerate}
\usepackage{mathptmx}
\usepackage{caption}
\usepackage{subcaption}

\usepackage{tikz}
\usepackage{tkz-base}
\usepackage{tkz-euclide}
\usetikzlibrary{patterns}

\newtheorem{introtheorem}{Theorem}

\newtheorem{theorem}{Theorem}[section]
\newtheorem{lemma}[theorem]{Lemma}
\newtheorem{conjecture}[theorem]{Conjecture}
\newtheorem*{conjecture*}{Conjecture}

\newtheorem{proposition}[theorem]{Proposition}
\newtheorem{corollary}[theorem]{Corollary}

\newtheorem{manualtheoreminner}{Theorem}
\newenvironment{manualtheorem}[1]{%
	\renewcommand\themanualtheoreminner{#1}%
	\manualtheoreminner
}{\endmanualtheoreminner}

\theoremstyle{definition}
\newtheorem{definition}[theorem]{Definition}
\newtheorem{example}[theorem]{Example}
\newtheorem{notation}[theorem]{Notation}

\theoremstyle{remark}
\newtheorem{remark}[theorem]{Remark}

\newcommand{\bbP}{\ensuremath{\mathbb{P}}}

\newcommand{\bbZ}{\ensuremath{\mathbb{Z}}}

\newcommand{\bbQ}{\ensuremath{\mathbb{Q}}}
\newcommand{\bbR}{\ensuremath{\mathbb{R}}}
\newcommand{\bbN}{{\mathbb{N}}}

\newcommand{\cO}{\ensuremath{\mathcal{O}}}

\newcommand{\NS}{\ensuremath {{N^1}}}

\newcommand{\ord}{\ensuremath{\operatorname {ord}}}

\newcommand{\codim}{\ensuremath{\operatorname {codim}}}
\newcommand{\supp}{\ensuremath{\operatorname {supp}}}

\newcommand{\Div}{\ensuremath{\operatorname {Div}}}
\newcommand{\Num}{\ensuremath{\operatorname {Num}}}

\newcommand{\Neg}{\ensuremath{\operatorname {Neg}}}
\newcommand{\Null}{\ensuremath{\operatorname {Null}}}
\newcommand{\Fix}{\ensuremath{\operatorname {Fix}}}

\newcommand{\Y}{\ensuremath{Y_{\bullet}}}
\newcommand{\subn}{_{\!\!\mathscr{N}}}

\numberwithin{figure}{section}

\begin{document}

\title{Newton--Okounkov bodies and Picard numbers on surfaces}

\author{Julio-José Moyano-Fernández} 
\address{Universitat Jaume I, Campus de Riu Sec, Departamento de Matem\'aticas \& Institut Universitari de Matem\`atiques i Aplicacions de Castell\'o, 12071
Caste\-ll\'on de la Plana, Spain} \email{moyano@uji.es}

\author{Matthias Nickel} 
\address{Johann Wolfgang Goethe-Universit\"at, Institut f\"ur Mathematik, Robert-Mayer-Stra\ss e 6-8, D-60325 Frankfurt am Main, Germany} \email{nickel@math.uni-frankfurt.de}

\author{Joaquim Roé}
\address{Universitat Aut\`{o}noma de Barcelona, Departament de Matem\`{a}tiques,
08193 Bellaterra (Barcelona), Spain} \email{jroe@mat.uab.cat}

\thanks{During the elaboration and revision of this work, the first author was partially funded by MCIN/AEI/10.13039/501100011033, by “ERDF A way of
making Europe” and by “European Union NextGeneration EU/PRTR”, Grants PID2022-138906NB-C22 and
TED2021-130358B-I00, as well as by Universitat Jaume I, Grants UJI-B2021-02 and GACUJIMA-2023-06. The second author was partially supported by the LOEWE grant ``Uniformized Structures in Algebra and Geometry''. The third author was been partially supported by MCIN/AEI/10.13039/501100011033 grants MTM2016-75980-P and PID2020-116542GB-I00, and by the Direcció General de Recerca Grant 2021 SGR 01468.}

\begin{abstract}
We study the shapes of all Newton--Okounkov bodies $\Delta_{v}(D)$ of a given big divisor $D$ on a surface $S$ with respect to all rank 2 valuations $v$ of $K(S)$. 

We obtain upper bounds for, and in many cases we determine exactly, the possible numbers of vertices of the bodies $\Delta_{v}(D)$. The upper bounds are expressed in terms of Picard numbers and they are \emph{birationally invariant}, as they do not depend on the model $\tilde{S}$ where the valuation $v$ becomes a flag valuation. 

We also conjecture that the set of all Newton--Okounkov bodies of a single ample divisor $D$ determines the Picard number of $S$, and prove that this is the case for Picard number 1, by an explicit characterization of surfaces of Picard number 1 in terms of Newton--Okounkov bodies.
\end{abstract}

\maketitle

\section{Introduction}

Newton--Okounkov bodies were introduced by A.~Okounkov \cite{Ok96} as a tool in representation theory; later, K.~Kaveh and A.~Khovanskii \cite{KaKho10,KaKho12} on the one hand, and R.~Lazarsfeld and M.~Musta{\c t}{\u a} \cite{LM09} on the other, developed a general theory with applications to both convex and algebraic geometry. Very quickly they gained a central role in the asymptotic theory of linear series on projective varieties, and they have been used to address questions of arithmetic geometry, combinatorics, Diophantine approximation, mirror symmetry, and representation theory (see S.~Boucksom's excellent exposition \cite{Bou14} for the general theory; Blum and Jonsson \cite{BJ20}, Boucksom and Chen \cite{BC11}, Fang, Fourier and Littelmann \cite{FFL17}, Harada and Kaveh \cite{HK15}, or Rietsch and Williams \cite{RW19}, are a sample of recent applications). 

To every big divisor $D$ on a normal projective variety $X$ of dimension $d$, and to every valuation $v$ of the field $K(X)$ of rank $d$, one associates a convex body in $\bbR^d$, called the Newton--Okounkov body $\Delta_{v}(D)$ of $D$ with respect to $v$; it contains all normalized values of divisors in the complete linear systems $|kD|$, $k\gg0$.
Its volume agrees (up to a factor of $d!$) with the volume of $D$ for every $v$, but its shape depends on the valuation $v$. 
A lot of effort has been put into trying to understand this dependence, and more generally the information encoded in the shapes of Newton--Okounkov bodies (see K\"uronya, Lozovanu and Maclean \cite{KLM12}, Ciliberto et al. \cite{CFKLRS}, Dumnicki et al. \cite{DHKRS}, Galindo et al. \cite{GMMFN}, and Ro\'e \cite{Roe16}).
By works of Lazarsfeld-Musta\c{t}\u{a} \cite{LM09} and S.~Y.~Jow \cite{Jow}, we know that the collection of all Newton--Okounkov bodies of $D$ is a complete numerical invariant of $D$, and by work of A.~Küronya, V.~Lozovanu and C.~Maclean \cite{KLM12}, we know that on a surface $S$ all Newton--Okounkov bodies are polygons.

Every valuation $v$ of maximal rank of the field $K(S)$ can be built from a smooth flag of subvarieties on some (smooth) birational model $\tilde{S}\rightarrow S$ of $S$, and it is often simpler to focus on flags living on $S$ itself: the proof of polygonality of 2-dimensional Newton--Okounkov bodies, for instance, relies on the flag construction (and Zariski decompositions).
It follows from the proof that the number of vertices of a Newton--Okounkov polygon with respect to a smooth flag on $S$ is at most $2\rho(S)+2$, where $\rho(S)$ is the Picard number of $S$.
However, this bound is often not sharp; in \cite{RS}, the third author and T.~Szemberg determined the number of vertices a Newton--Okounkov polygon may have, where $D$ is required to be ample, $S$ is smooth, and the valuation is induced by a smooth flag on $S$. 
\medskip

The main thrust of the present work is to showcase the benefits of considering valuations determined by flags on higher models $\tilde{S}$ and not just on $S$.
Although this does often lead to slightly more involved arguments, one in fact gains clarity as statements become simpler; this is the case for instance in our first main result. 
It could be expected that, since higher smooth models $\tilde{S}\rightarrow S$ have larger Picard number, the number of vertices of Newton--Okounkov polygons with respect to flags lying on $\tilde{S}$ could grow; somewhat surprisingly, this is not the case:

\begin{introtheorem}\label{upperbound-intro}
Let $S$ be a smooth projective algebraic surface, and let $D$ be a big Cartier divisor on $S$. For every rank 2 valuation $v$ of $K(X)$, the Newton--Okounkov body $\Delta_{v}(D)$ has at most $2\rho(S)+2$ vertices.
\end{introtheorem}

There are geometric restrictions which usually make it impossible to attain the maximal number of vertices $2\rho(S)+2$ with flags lying on $S$, see \cite{RS}; these restrictions loosen when we consider arbitrary valuations of maximal rank, and we in fact hope they disappear:

\begin{introtheorem}\label{existence-intro}
Let $S$ be a smooth projective algebraic surface with Picard number $\rho(S)=1$, and let $D$ be an ample Cartier divisor on $S$. There exist rank 2 valuations $v_3, v_4$ of $K(S)$, such that the Newton--Okounkov body $\Delta_{v_k}(D)$ has exactly $k$ vertices.	
	
Conversely, if $S$ is a smooth projective algebraic surface with Picard number $\rho(S)>1$, then there exist rank 2 valuations $v$ such that $\Delta_v(D)$ has at least five vertices.
\end{introtheorem}

In fact, we propose the following:

\begin{conjecture*}\label{conjecture-intro}
Let $S$ be a smooth projective algebraic surface, and let $D$ be an ample Cartier divisor on $S$. For every natural number $k$, $3\le k \le 2\rho(S)+2$, there exists a rank 2 valuation $v$ of $K(S)$, such that the Newton--Okounkov body $\Delta_{v}(D)$ has exactly $k$ vertices.	
\end{conjecture*}

Theorem \ref{existence-intro} on one hand shows that the upper bound can always be attained if the Picard number is 1 (generalizing results of \cite{CFKLRS,GMMFN} for $\bbP^2$) and on the other hand characterizes surfaces with Picard number 1 in terms of Newton--Okounkov bodies.
Our conjecture means that we expect both facts to generalize for arbitrary Picard number, and so that Newton--Okounkov bodies determine Picard numbers.
\medskip

One sees in Theorem \ref{upperbound-intro} that if $D$ is the pullback of a divisor by a birational morphism $S\rightarrow S'$ of smooth projective surfaces, then the Picard number of $S'$ (smaller than that of $S$) may be used to bound the number of vertices of Newton--Okounkov bodies of $D$. This suggests that curves contracted by the maps $S\overset{|kD|}{\dashrightarrow}\bbP^{N}$ (which are birational to the image for large $k$ and big $D$) ``should not count'' in the Picard number upper bound; the appropriate notion is the following.

\begin{definition}
Let $S$ be a smooth projective surface, and $D$ a big and nef Cartier divisor on $S$. We define the Picard number of $S$ relative to $D$, or simply the Picard number of $D$, written $\rho_D(S)$, as the Picard number of $S$ minus the number of irreducible curves $C$ on $S$ with $C\cdot D=0$. (If $D$ is big but not nef, we define $\rho_D(S)$ as $\rho_{P_D}(S)$ where $P_D$ is the positive part in the Zariski decomposition of $D$, see Definition~\ref{def:rho-D}). 
\end{definition}

In Section \ref{sec:prerequisites} we show that this Picard number is invariant under pullback by birational morphisms; this allows us to extend the notion to singular surfaces using resolution of singularities. 
Then we can prove a stronger version of Theorem \ref{upperbound-intro},

\begin{manualtheorem}{\ref*{upperbound-intro}\#}\label{upperbound-sharp-intro}
Let $S$ be a normal projective algebraic surface, and let $D$ be a big Cartier divisor on $S$. Then for every rank 2 valuation $v$ on $K(S)$, the Newton--Okounkov body $\Delta_v(D)$ has at most $2\rho_D(S)+2$ vertices.
\end{manualtheorem}

It is plausible that the conjecture above holds for big divisors (not necessarily ample) replacing $\rho$ by $\rho_D$, but we have little evidence to support this assertion. 
\medskip

In the proofs of Theorems \ref{upperbound-sharp-intro} and \ref{existence-intro} we use the description of Newton--Okounkov polygons in terms of Zariski decompositions introduced by Lazarsfeld--Musta\c{t}\u{a} in \cite{LM09} and we also employ techniques from \cite{KLM12} and \cite{RS}.
However, these techniques do not immediately carry over to valuations whose flags live in higher models $\tilde{S}\rightarrow S$. One important fact that comes to our help is that \emph{curves orthogonal to $D$ never contribute vertices} to a polygon $\Delta_{v}(D)$, cf. Lemma \ref{lem:pullback}.
We also investigate Newton--Okounkov bodies with respect to smooth flags on a fixed model, but since our objective is to apply this study to higher models  $\tilde S \rightarrow S$, it becomes essential to pay attention to ``contracted curves'', i.e., irreducible curves whose intersection number with the given big divisor $D$ is not positive. 
Flags whose divisorial part does intersect $D$ positively turn out to behave in a way similar to smooth flags, which allows us to extend some results of \cite{RS} for ample divisors to the general setting of big divisors, with suitable modifications to take care of contracted curves.
We attach to every big divisor $D$ on a smooth surface $S$ a number $mv(D)$, computed in terms of configurations of negative curves (Definition \ref{def:mv}), and invariant under pullbacks, which allows us to  describe all Newton--Okounkov bodies of $D$ whose flag lives on a given model and intersects $D$ positively:

\begin{introtheorem}\label{positive-flags-intro}
On every smooth projective surface $S$, for every big divisor $D$ and every integer $\delta$ with $3\le \delta \le mv(D)$, there exists a smooth flag $\Y$ in $S$ such that the Newton--Okounkov polygon \(\Delta_{\Y}(D)\) has exactly \(\delta\) vertices.
	
On every normal projective surface $S$, for every big and nef divisor $D$, there is a model $\tilde{S}\rightarrow S$ and a flag $\Y=\{\tilde{S}\supset C \supset \{p\}\}$ such that $D\cdot \pi(C)>0$ and $\Delta_{\Y}$ is a $k$-gon if and only if $3\le k \le mv(D)$. 
\end{introtheorem} 
If $D$ is big but not nef, a version of the second statement still holds, only replacing $D\cdot\pi(C)$ by $P_D\cdot\pi(C)$, where $P_D$ stands for the positive part in the Zariski decomposition $D=P_D+N_D$.

The number $mv(D)$ is usually smaller than $2\rho_D(S)+2$, so there are still geometric restrictions making it impossible to reach the maximum number of vertices; this can only be achieved by using flags (on some model $\tilde{S}\rightarrow S$) whose divisorial part does not intersect (the pullback of) $D$. 
These \emph{contracted}, or \emph{infinitesimal flags} are extremely interesting, and are the ones allowing us to prove Theorem \ref{existence-intro}, but it is also challenging to work with them; in particular, it seems that the path to proving the conjecture above should go through the construction of adequate infinitesimal flags, which we were unable to do at this point.

It follows from our analysis that there are important qualitative differences between the Newton--Okounkov bodies of a fixed divisor $D$ with respect to flags whose divisorial part has zero or non-zero intersection with $D$. 
We would like to point out that such a distinction appeared already in the work of Choi, Park and Won \cite[Theorem 1.1]{CPW17}, although we are not aware of any direct connection between both phenomena, since they work with Newton--Okounkov bodies of non-big pseudoeffective divisors, which are just segments and thus the question of the number of points is moot for them.

The paper is organized as follows. In Section \ref{sec:prerequisites} we recall necessary notions and relevant facts from existing literature. Section~\ref{sec:upper-bound} is devoted to the proof of Theorem \ref{upperbound-sharp-intro}. In Section \ref{sec:positive-flags} we study Newton--Okounkov bodies with respect to flags whose divisorial part intersects $D$ positively and prove Theorem \ref{positive-flags-intro}. In Sections \ref{sec:examples} and \ref{sec:existence} we study Newton--Okounkov bodies with respect to flags whose divisorial part does not intersect $D$; first we exhibit examples that show the differences in behaviour with respect to positively intersecting flags, and finally we prove Theorem \ref{existence-intro}.

We work over the complex field. 


\section{Newton--Okounkov bodies on surfaces}
\label{sec:prerequisites}

In this section we recall the terminology used throughout the paper, according to the conventions in \cite{Roe16}, \cite{CFKLRS}, or \cite{GMMFN}. For a general reference on positivity the reader is referred to \cite{La1, La2}.

Let $X$ be a normal, projective, algebraic variety of dimension $n$ defined over the complex numbers with function field $K(X)$. 
The group of Cartier divisors on $X$ will be denoted as $\Div(X)$, whereas $\Num(X)$ will be the subgroup of divisors numerically equivalent to 0, and $\NS(X)=\Div(X)/\Num(X)$ will be the Néron-Severi group of $X$.
Let us recall \cite[1.1.16]{La1}: the N\'eron-Severi group is a free abelian group of finite rank $\rho(X)$, the Picard number of $X$.

A Cartier divisor $D$ on $X$ is called \emph{big} if $h^0(X,\mathcal{O}_X(kD))$ grows like $k^n$. This is equivalent to say that $\mathcal{O}_X(D)$ has maximal Iitaka dimension. Bigness makes sense also for $\mathbb{Q}$- and $\mathbb{R}$-divisors, and it only depends on the numerical equivalence class of $D$, see \cite[(0.3)]{LM09}. 
The convex cone in $\NS(X)_{\bbR}$ of all big $\mathbb{R}$-divisor classes on $X$ is called the big cone, and it will be denoted by $\mathrm{Big}(X)$.

Let $v:K(X)\setminus \{0\} \to \mathbb{Z}^n_{\mathrm{lex}}$ be a discrete valuation of rank $n$ on $X$ whose value group is $\mathbb{Z}^n$ endowed with the lexicographical ordering. For any such a valuation and any big divisor $D$ on $X$ one defines the Newton--Okounkov body of $D$ with respect to $v$ as
\[ 
\Delta_v(D) = \overline{ \left\{ \frac{v(s)}{k} : s \in H^0(k D)\subset K(X), k \in \bbN \right\} } \,,
\]
where $\overline{\{~ \cdot ~ \}}$ stands for the closure with respect to the usual topology of $\mathbb{R}^n$.

The simplest valuations of maximal rank, and the most used in the context of Newton--Okounkov bodies, are defined via flags. Let us consider a flag of irreducible subvarieties 
\[
\Y = \left\{X = Y_0 \supset Y_1 \supset \dots \supset Y_n =\{p\}\right\}.
\]
The flag $\Y$ is said to be \emph{full} and \emph{admissible} if $Y_i$ has codimension $i$ in $X$ and is smooth at the point $p$. The point $p$ is called the \emph{center} of the flag.
It is possible to construct a discrete valuation of rank $n$ from $\Y$ as follows. For every $\phi \in K(X)$ we set $\phi_0 = \phi$, and
\[ 
v_i (\phi)  = \ord_{Y_i}(\phi_{i-1}), \quad \phi_i = \left.\frac{\phi_{i-1}}{g_i^{v_i(\phi)}} \right|_{Y_i} \, \ \ \mbox{ for } \ i=1,\ldots, n,
\] 
where $g_i$ is a local equation of $Y_i$ in $Y_{i-1}$ around $p$. It is clear that $v_{\Y} = (v_1,\dots,v_n)$ defines a rank $n$ discrete valuation $K(X)\setminus \{0\} \rightarrow \bbZ^n_\text{lex}$. Now we can define the Newton--Okounkov body of $D$ with respect to the flag $\Y$  as $\Delta_{\Y}(D) := \Delta_{v_{\Y}}(D)$.

This definition can be extended to the case of an admissible flag on a birational model $\tilde{X}$ of $X$ where $\pi: \tilde{X} \rightarrow X$ is a proper birational morphism. Such a flag 
\[
\Y = \left\{\tilde{X} = Y_0 \supset Y_1 \supset \dots \supset Y_n =\{p\}\right\}
\]
is called \emph{infinitesimal} if $\pi(Y_1)=\pi(p)$ and \emph{proper} if $\codim \pi(Y_i) = i$. If $\tilde{X}=X$ and $\pi=\mathrm{id}_X$, then the flag is said to be \emph{smooth} at $p$. The Newton--Okounkov bodies given by infinitesimal resp. proper resp. smooth flags are called infinitesimal resp. proper resp. smooth.
By \cite[Theorem 2.9]{CFKLRS}, every discrete valuation of rank $n$, and hence every Newton--Okounkov body, is of one of these types.

We are interested in the case $n=2$. 
Let $S$ be therefore a complex projective normal surface.

\begin{remark}
\label{modelRemark}
For the computations of Newton--Okounkov bodies, we shall use the Zariski decomposition of effective divisors, which needs $S$ to be smooth; this is not really restrictive, because passing to a resolution of singularities $\pi:\tilde{S}\to S$, one has $\Delta_v(D)=\Delta_v(\pi^*(D))$. 
On the other hand, by \cite[Theorem 2.9]{CFKLRS} for every rank 2 valuation $v$ there is a model $\pi:\tilde S\to S$ and a smooth flag $\Y=\{\tilde{S}\supset C \supset \{p\}\}$ on $\tilde{S}$ such that $v$ is the valuation associated to $\Y$. Hence we may also assume that $v$ is a flag valuation.
Note that the choice of flag is uniquely determined on every model where it exists: if $\pi':\tilde{S}'\to \tilde{S}$ is a higher model, then the restriction of $\pi'$ to the strict transform $C'$ of $C$ in $\tilde{S}'$ is an isomorphism, so we have a uniquely defined flag $\Y'=\{\tilde{S}'\supset C'\supset\{p'\}\}$ on $\tilde{S}'$ inducing the same valuation, namely the one given by the unique point $p'$ on $C'$ above $p$.

\end{remark}

When $S$ is smooth, every big $\mathbb{Q}$-divisor $D$ admits a Zariski decomposition $D=N_D+P_D$, where $P_D$ (the positive part of $D$) is a nef $\mathbb{Q}$-divisor and $N_D$ (the negative part of $D$) is an effective $\mathbb{Q}$-divisor with the property that, whenever $kD$ and $kN_D$ are integral divisors, multiplication by the section defining $kN_D$ induces an isomorphism
\[
H^0(S,\mathcal{O}(kP_D)) \longrightarrow H^0(S,\mathcal{O}(kD)).
\]
Zariski decompositions can be also defined for big $\mathbb{Q}$- and $\mathbb{R}$-divisors. They are determined by the numerical equivalence class of $D$, so it makes sense to study them as functions on the N\'eron-Severi space \cite{BKS}.

The first explicit examples of Newton--Okounkov bodies appeared in \cite[Section 6]{LM09}. In fact, \cite[Lemma 6.3 and Theorem 6.4]{LM09} gives a description of those bodies associated to any big divisor on $S$, in terms of Zariski decompositions:

\begin{proposition}[Lazarsfeld--Musta\c{t}\u{a}]\label{pro:LM-alphabeta}
Let $D$ be a big divisor on a smooth projective surface $S$ and let $\Y=\left\{ S \supset C \supset \left\{ p \right\} \right\}$ be an admissible flag on $S$. Let $D_t=P_t+N_t$ be the Zariski decomposition of $D_t=D-t C$, $t\in\bbR$. Let $\nu$ be the coefficient of $C$ in $N_0$, and let $\mu = \sup\{t>0 : D - t C \mbox{ is big}\}>\nu\ge 0$. Then the Newton--Okounkov body of $D$ with respect to $\Y$ is given by
\[
\Delta_{\Y} = \{(t,y) \in \bbR^2 : \nu \leq t \leq \mu, \, \alpha(t) \leq y \leq \beta(t) \} \,,
\]
where
\[
\alpha(t)=\ord_p(N_t|_C), \quad \beta(t) = \alpha(t)+(C \cdot P_t).
\]
\end{proposition}

\begin{remark}
Since by \cite[Theorem 2.9]{CFKLRS} every rank 2 valuation on a surface $S$ is a flag valuation on some model $\tilde{S}\rightarrow S$, the previous proposition holds for every Newton--Okounkov body $\Delta_{v}(D)$, provided that the Zariski decomposition is done on $\tilde{S}$.
\end{remark}
In the study of Zariski decompositions, Bauer, Küronya and Szemberg introduced in \cite{BKS} the concept of Zariski chamber as a subcone building a partition of the big cone
$$
\mathrm{Big}(X)=\bigcup_{P ~\mathrm{big}~ \wedge~ \mathrm{nef}} \Sigma_P,
$$
where
\begin{align*}
\Sigma_P &= \{D \in \mathrm{Big}(X) : \mathrm{Neg}(D)=\mathrm{Null}(P)\}\\
\mathrm{Neg}(D)&=\{C : C \mbox{ is an irreducible component of } N_D\} \\
\mathrm{Null}(D)&=\{C: C \mbox{ is an  irreducible curve with } C \cdot D = 0\}
\end{align*}
Note that the definitions give $\Neg(D)\cup\Null(D)\subset\Null(P_D)$ (the union may be disjoint or not). By exploiting the Zariski chamber decomposition it is not hard to prove that Zariski decompositions depend continously on the big divisor $D$:

\begin{proposition}[{\cite[Prop. 1.16]{BKS}}]\label{pro:continuity-zariski}
	Let $(D_n)$ be a sequence of big divisors converging in $N^1(X)_{\bbR}$ to a big
	divisor $D$. If $D_n=P_n+N_n$ is the Zariski decomposition of $D_n$, and if $D=P+N$ is the
	Zariski decomposition of $D$, then the sequences $(P_n)$ and $(N_n)$ converge to $P$ and $N$ respectively.
\end{proposition}

This fact guarantees, for instance, that the functions $\alpha$ and $\beta$ above are continuous, and it will be essential for our arguments later on.

\begin{definition}\label{def:rho-D}
	Let $D$ be a big divisor on a projective algebraic surface $S$.
	We define the \emph{Picard number of $S$ relative to $D$}, if $S$ is smooth, as $\rho_D(S)=\rho(S)-\#\Null(P_D)$. 
	
	The definition can be extended to the case of a possibly singular normal surface $S$, as follows. 
	Observe first that the definition for smooth surfaces is invariant under pullback by proper birational morphisms: if $S$ is smooth and $\pi:\tilde{S}\to S$ is the blowup centered at a point $p$, then $\pi^*(P_D)=P_{\pi^*(D)}$, and by the projection formula \cite[8.1.7]{Ful98},
	\begin{equation*}
	\Null(\pi^*(P_D))=\{\tilde{C} \,:\, C\in \Null(D)\} \cup \{E_p\},
	\end{equation*}
	therefore 
	\begin{equation}\label{eq:rho-birational}
	\#\Null(\pi^*(D))-\#\Null(D)=1=\rho(\tilde{S})-\rho(S).
	\end{equation}
	It follows that $\rho_{\pi^{*}D}(\tilde{S})=\rho_D(S)$.
	Now, if $D$ is a big and nef divisor on a normal surface $S$, we may define $\rho_D(S)$ as $\rho_{\pi^*(D)}(\tilde{S})$ for any resolution of singularities $\pi:\tilde{S}\to S$; equation \eqref{eq:rho-birational} guarantees that this definition does not depend on the choice of a resolution.
\end{definition}	

\begin{remark}\label{rem:null-defneg}
	When $S$ is smooth, as $D$ is big, $P_D$ is big as well, and in particular $P_D^2>0$. The Index Theorem implies that the intersection form restricted to $P_D^\perp\supset\langle\Null(P_D)\rangle$ is negative definite, and the number of curves $\#\Null(P_D)$ is bounded above by $\rho(S)-1$.
	Therefore $1\le\rho_D(S)\le\rho(S)$, with $\rho_D(S)=\rho(S)$ if and only if $D$ is ample.
\end{remark}

The description of Newton--Okounkov bodies of Proposition \ref{pro:LM-alphabeta} was further investigated by K\"uronya, Lozovanu and Maclean in \cite{KLM12} and by Roé and Szemberg in \cite{RS}; we summarize their results as the following theorem:

\begin{theorem}[K\"uronya--Lozovanu--Maclean, Roé--Szemberg]\label{pro:KLM}
Let $D$ be a big divisor on a smooth surface $S$, and consider a smooth flag $\Y=\{S\supset C \supset \{p\} \}$. Let $D_t=D-tC$ be with Zariski decomposition $D_t=P_t+N_t$ as above, and let $\alpha(t)$, $\beta(t)$ be the functions determining the lower and upper boundaries of $\Delta_{\Y}(D)$ given by Proposition \ref{pro:LM-alphabeta}. Then,
	\begin{enumerate}
		\item $\alpha(t)$ is a nondecreasing function. \cite[Theorem B]{KLM12}
		\item $\alpha(t)$ and $\beta(t)$ are piecewise linear functions with finitely many pieces. \cite[Theorem B]{KLM12}
		\item The function $\alpha(t)$ or $\beta(t)$ has a point of non-differentiability at $t=t_0$ if and only if the negative part of $D_t=D-tC$ acquires one or more irreducible components at $t_0$, i.e., for every $\varepsilon>0$ the supports of $N_{t_0}$ and $N_{t_0+\varepsilon}$ differ. \label{item3}
		\item The function $\alpha(t)$ has a point of non-differentiability at $t=t_0$ if and only if one or more irreducible components acquired at $t_0$ are contained in the connected component of $N_{t_0+\varepsilon}$ that passes through $p$. \cite[Proposition 3.3]{RS} \label{item4}
		\item The function $\beta(t)$ has a point of non-differentiability at $t=t_0$ if and only if one or more irreducible components acquired at $t_0$ are contained in a connected component of $N_{t_0+\varepsilon}$ that intersects $C$ away from $p$. \cite[Proposition 3.3]{RS} \label{item5}
	\end{enumerate}
\end{theorem}

\begin{proof}
The only statement that is not immediately clear from the references is (\ref{item3}). By (\ref{item4}) and (\ref{item5}) we only need to show that if all irreducible components acquired at $t_0$ do not pass through $p$, then one or more of them is contained in a connected component of $N_{t_0+\varepsilon}$ that intersects $C$ away from $p$. Finally, this is true since otherwise the support of $N_{t_0}$ and $N_{t_0+\varepsilon}$ would not change.
\end{proof}

The theorem immediately extends to a normal surface as long as $\Y$ is a smooth flag, by computing the Zariski decomposition of the pullback of $D$ to any resolution of singularities of $S$.

The fact that $N_t$ acquires a new irreducible component at $t_0$ can be equivalently expressed by saying that $D_t$ \emph{crosses a Zariski wall} between chambers, i.e., $D_{t_0}$ and $D_{t_0+\varepsilon}$ belong to different cones $\Sigma_P$ for every $\varepsilon>0$.

These results were then used to connect the number of vertices of the Newton--Okounkov body with the Picard number of the surface.

\begin{corollary}\label{cor:2rho+1}
Let $S$ be a smooth surface, $D$ a big divisor on $S$ and $\Y=\{S\supset C \supset\{p\}\}$ a smooth flag. 
Let $D_t=D-tC$ be with Zariski decomposition $D_t=P_t+N_t$ as above,
and let $\mu = \sup\{t>0 : D_t\text{ is big}\}$.

Then $\Delta_{\Y}(D)$ is a polygon with at most $2\rho(S)+1$ vertices.
If $N_\mu$ has $\rho(S)-1$ irreducible components, then $\Delta_{\Y}(D)$ has a single rightmost vertex.
\label{cor:vertices}
\end{corollary}

In the statement of Corollary \ref{cor:2rho+1}, one can ``almost'' replace $\rho(S)$ by $\rho_D(S)$, as we shall see in Theorem \ref{thm:bound}. This has two advantages; first, and most obvious, the bound is sharper for non-ample $D$, and second, $\rho_D(S)$ is invariant under proper birational morphisms $\pi$ (replacing $D$ by $\pi^*(D)$).
We include a short sketch of the proof of Corollary \ref{cor:2rho+1}, since we will use the ideas therein extensively.

\begin{proof}
	The description of the Newton--Okounkov body in Proposition \ref{pro:LM-alphabeta} together with the decomposition of the big cone into Zariski chambers \cite{BKS} show that the number of vertices of $\Delta_{\Y}(D)$ is bounded by two times the number of Zariski walls that $D_t = D - t C$ crosses for $0\le\nu < t < \mu$ and the leftmost and rightmost vertices (at most two of each).
	
	Let us prove the second claim first.
	If $N_\nu$ consists of $\rho-1$ irreducible components, then by the Index Theorem the divisor $P_\mu$, which is orthogonal to all components of $N_\mu$, must be a multiple of a big class. Since $P_\mu$ is not big, it has to be zero, and so $\Delta_{\Y}(D)$ has only one rightmost vertex. 
	
	Next, the minimality of the negative part in the Zariski decomposition ensures that $\supp(N_t) \subset \supp(N_\mu)$ for $0 \leq \nu\leq t \leq \mu$. As $N_\mu$ has negative definite intersection matrix, by the Index Theorem the support of $N_\mu$ can only contain $\rho-1$ curves at most, and therefore at most $\rho-1$ Zariski walls are crossed. 
	Therefore $\Delta_{\Y}(D)$ is a polygon with at most $2(\rho-1)+4=2\rho(S)+2$ vertices.	
	However, if $\Delta_{\Y}(D)$ had exactly $2\rho(S)+2$ vertices, $N_\nu$ would consist of $\rho-1$ irreducible components, and therefore $\Delta_{\Y}(D)$ would have only one rightmost vertex, so $\Delta_{\Y}(D)$ has at most $2\rho(S)+1$ vertices, as claimed.
\end{proof}

Moreover, \cite{RS} show that the bound is attained, i.e., for every positive integer $\rho$ there exist a smooth projective surface $S$, a big divisor $D$ and a smooth flag $\Y$ such that $\Delta_{\Y}(D)$ has $2\,\rho(S)+1$ vertices.
To do so, they determine the numbers of vertices that can be obtained as Newton--Okounkov bodies of every \emph{ample} divisor $D$.
We can partially extend these results of \cite{RS} to the case of non-ample divisors, applying techniques from the theory of stable base loci on surfaces (see \cite{BKS}, \cite{Zar62}). 

Given an effective divisor $D$, let $\Fix(D)$ denote the fixed divisor of $|D|$ (i.e., every $D'$ linearly equivalent to $D$ satisfies $D'\ge \Fix(D)$ and the base locus of $|D-\Fix(D)|$ is zero-dimensional). Let us denote by $SB(D)$ the stable base locus of $D$, i.e., the intersection of the base loci of $|kD|$ for all $k\ge 1$.
We also use the notation $\supp(D)$ for the support of an effective divisor $D$.
The following lemma collects a few known facts:
\begin{lemma}\label{lem:sb-nef}
	Let $D$ be a big divisor on a smooth surface $S$ and let $D=P_D+N_D$ be its Zariski decomposition.	Then
	\begin{enumerate}
		\item The stable base locus $SB(D)$ decomposes as $SB(D)=\supp(N_D)\cup SB(P_D)$.
		\item The stable base locus $SB(P_D)$ is contained in $\supp(\Fix(P_D))$, and it has no isolated points.
		\item The stable base locus $SB(P_D)$ is the union of some of the components of $\Null(P_D)$.
		\item \label{lem:sb-nef-bounded}For every component $E$ of $SB(P_D)$, the coefficient of $E$ in $\Fix(kP_D)$ is bounded, i.e., there exists $\ell \in \bbN$ such that for every component $E$ of $SB(P_D)$, and every $k\ge 0$, the coefficient of $E$ in $\Fix(kP_D)$ is less than $\ell$.
	\end{enumerate}
\end{lemma}

All these facts are due to Zariski \cite{Zar62}. See \cite[Section 2]{BKS}, especially Proposition 2.5 and its proof, for a modern account with notations consistent with ours.

\begin{notation}\label{notation}
	Let $S$ be a projective normal surface, and assume a big divisor $D$ and an admissible flag $\Y=\{S\supset C\supset\{p\}\}$ are given. 
	We fix the following notations for the rest of the paper.
	\begin{itemize}
		\item[$\diamond$] $D_t=D-t C$ for every $t\in \bbR$.
		\item[$\diamond$] $P_t$ resp.~$N_t$ is the positive resp.~negative part of the Zariski decomposition $D_t=P_t+N_t$ of $D_t$.
		\item[$\diamond$] $\nu=\nu_C(D)\in \bbQ$ is the coefficient of $C$ in $N_0$.
		\item[$\diamond$] $\mu = \mu_C(D) =  \sup\{t>0 : D - t C \mbox{ is big}\}$.
		\item[$\diamond$] For every (irreducible, reduced) curve $C_i$ in the support of $N_\mu$, $$t_{C_i}=t_{C_i}(D,C)=\inf \{t\in \bbR : \ord_E(N_t)>0\}.$$ 
		Note that the support of every $N_t$ with $0\le t\le \mu$ is contained in the support of $N_\mu$.
		\item[$\diamond$] A vertex $p=(t,s)$ of the polygon $\Delta_{\Y}(D)$ is called \emph{leftmost}, \emph{rightmost} or \emph{interior} when $t=\nu$, $t=\mu$ or $\nu<t<\mu$, respectively.
	\end{itemize}
\end{notation}


\section{Upper bound for the number of vertices}
\label{sec:upper-bound}

One might expect that Newton--Okounkov bodies with respect to flags living in birational models of $S$ with larger Picard number could have more vertices. However, extensive computations on blowups of $\bbP^2$ (see \cite{CFKLRS, GMMFN}) have shown that this is not the case for $S=\bbP^2$ and suggest that a bound similar to the ones above might hold not only for smooth flags, but for arbitrary Newton--Okounkov polygons. 
In this section we prove Theorem \ref{upperbound-sharp-intro}, to this effect.
Simultaneously, we shall obtain a sharp bound that extends the bound of \cite{RS} for ample divisor classes to arbitrary big divisor classes.

The main step towards the proof is the observation that, in the association of negative curves to vertices of the polygon, described above, those curves in $\Null(D)$ do not contribute vertices:


\begin{lemma}\label{lem:pullback}
	Let $D$ be a big and nef divisor on a smooth surface $S$, and consider the admissible flag $\Y:=\{S\supset C \supset \{p\}\}$ on $S$.
	Let us write as above $D_t=D-tC$ with Zariski decomposition $D_t=P_t+N_t$.
	Let $(t,s)$ be an interior vertex of $\Delta_{\Y}(D)$ (i.e., with $0<t<\mu_C(D)$).
	Then there exists a curve $F_i$, not belonging to $\Null(D)$ nor to $N_t$, with $t_{F_i}=t$.
\end{lemma} 
	
\begin{proof}
By Theorem \ref{pro:KLM} we know that the supports of $N_{t+\varepsilon}$ and $N_{t}$ differ. Let 
\[N_{t+\varepsilon}=a_1 F_1 + \dots+ a_k F_k\]
be the negative part of $D_{t+\varepsilon}=D-(t+\varepsilon)C$
for $|\varepsilon|$ small enough,
where $a_i=a_i(\varepsilon)$ is a continuous function, piecewise linear, with a point of non-differentiability at $\varepsilon=0$ (at least for \emph{some} $a_i$). Assume that  
$F_i,\dots, F_k$ are the irreducible components of $N_{t+\varepsilon}$ not contained in $N_{t}$, i.e., $a_j(0)=0$ if and only if $j\ge i$. 
We need to prove that at least one of these does not belong to $\Null(D)$. 

By continuity of the Zariski decomposition (see Proposition \ref{pro:continuity-zariski}), no curve $H$ satisfying $P_t\cdot H>0$ belongs to the negative part $N_{t+\varepsilon}$ for small $\varepsilon$. 
Therefore, the components $F_i,\dots, F_k$ all belong to $\Null(P_t)$.
Next we claim that at least one curve in $F_i,\dots, F_k$ intersects $C$ or a component of $N_t$.
Indeed, if we assume by way of contradiction that none of them does, then $N'_{t+\varepsilon}:=a_1F_1+\dots+a_{i-1}F_{i-1}$ satisfies
$$ (D_{t+\varepsilon}-N'_{t+\varepsilon})\cdot F_j=
\begin{cases}
(D_{t+\varepsilon}-N_{t+\varepsilon})\cdot F_j =0 & \text{for }j<i,\\
D\cdot F_j\ge 0 & \text{for }j\ge i; 
\end{cases}$$
and for any curve $E$ not among the $F_j$, we have that $(D_{t+\varepsilon}-N'_{t+\varepsilon})\cdot E \ge P_{t+\varepsilon}\cdot E \ge 0$, 
hence $D_{t+\varepsilon}-N'_{t+\varepsilon}$ is nef, contradicting the minimality of $N_{t+\varepsilon}$ in the Zariski decomposition.

So we can assume without loss of generality that $F_i$ intersects either $C$ or a component of $N_t$. Then $F_i$ does not belong to $\Null(D)$, because if it did, we would have
\[
0= P_{t}\cdot F_i=(D-tC-N_t)\cdot F_i =-(tC+N_t)\cdot F_i< 0,
\]
a contradiction.
\end{proof}

\begin{theorem}[Theorem {\ref{upperbound-sharp-intro}}]\label{thm:bound}
	Let $S$ be a normal projective algebraic surface, and $D$ a big  Cartier divisor on $S$. 
	Then for every rank 2 valuation $v$ on $K(S)$, the polygon $\Delta_{v}(D)$ has at most $2\,\rho_D(S) +2$ vertices.
\end{theorem}

\begin{proof}
Since $S$ is normal, for every resolution of singularities $\pi:\tilde{S}\rightarrow S$ one has isomorphisms 
$$
H^0(\mathcal{O}_{\tilde{S}}(k\,\pi^*(D)))\cong H^0(\mathcal{O}_{S}(k\,D))
$$
for all $k$. The pullback divisor $\pi^{*}(D)$ on $D$ is therefore big, 
and the valuations of $K(\tilde{S})=K(S)$ applied to sections of $\cO_S(k\,D)$ and $\cO_{\tilde{S}}(k\,\pi^*(D))$ are the same, so $\Delta_{v}(\pi^*(D))=\Delta_{v}(D)$. On the other hand, $\rho_{\pi^*(D)}(\tilde{S})=\rho_D(S)$ by definition. 
So we can replace $S$ with $\tilde{S}$ and assume that $S$ is smooth.

By the same token, since $v$ is given by a smooth flag $\Y=\{\tilde{S}\supset C \supset \{p\}\}$, on some blowup $\pi:\tilde{S}\rightarrow S$, we may replace $S$ with $\tilde S$ and assume that $v$ is the flag valuation of an admissible flag.

Now let $D=P_D+N_D$ be the Zariski decomposition of $D$. 
The body $\Delta_v(D)$ is a translate of the body $\Delta_v(P_D)$ (a well known fact which holds even in higher dimension, see \cite[Theorem C, 3]{KL17a} together with Remark \ref{modelRemark}). 
Moreover, $\rho_D(S)=\rho_{P_D}(S)$; so we can replace $D$ by $P_D$ and assume that $D$ is nef.


Write as above $D_t=D-tC$ and let $D_t=P_t+N_t$ be its Zariski decomposition for $t\in [0,\mu]$.
By continuity of the Zariski decomposition, $N_\mu$ has the same support as $N_{\mu-\varepsilon}$ for $\varepsilon>0$ small enough.
Since $P_{\mu -\varepsilon}$ is big and nef, by the Index Theorem the number of components of $\Null(P_{\mu-\varepsilon})$ (which includes all components of $N_\mu$) is at most $\rho(S)-1$.
For every component $E$ of $\Null(D)$, 
$$
0\le P_{\mu-\varepsilon}\cdot E=D\cdot E -{\mu-\varepsilon} C\cdot E-N_{\mu-\varepsilon}\cdot E=(-\mu-\varepsilon)\, C\cdot E-N_{\mu-\varepsilon}\cdot E,
$$
so $E$ is either equal to $C$, to a component of $N_{\mu-\varepsilon}$ or to a component of $\Null(P_{\mu-\varepsilon})$.
Since all components of $N_{\mu-\varepsilon}$ are components of $\Null(P_{\mu-\varepsilon})$, either $E=C$ or $E$ is a component of $\Null(P_{\mu-\varepsilon})$.

We now distinguish two cases. 
If $C$ is not a component of $\Null(D)$, the number of components of $N_\mu$ which are not components of $\Null(D)$ is bounded above by 
$$
\rho(S)-1-\#\Null(D)=\rho_D(S)-1, 
$$
so the number of interior vertices in $\Delta_v(D)$ is bounded above by $2\rho_D(S)-2$,
and the same argument as in the proof of Corollary \ref{cor:vertices} implies that the total number of vertices is at most $2\rho_D(S) +1$.
On the other hand, if $C=E$ belongs to $\Null(D)$, the number of components of $N_\mu$ which are not components of $\Null(D)$ is bounded above by 
$$
\rho(S)-1-(\#\Null(D)-1)=\rho_D(S), 
$$
and the number of interior vertices in $\Delta_v(D)$ is bounded above by $2\rho_D(S)$. Moreover, again the same argument as in the proof of Corollary \ref{cor:vertices} shows that, if the number of interior vertices equals $2\rho_D(S)$, then there is just one rightmost vertex, and since $C$ is a component of $\Null(D)$ we have $\beta(0)-\alpha(0)=D\cdot C=0$; therefore there is just one leftmost vertex, for a total number of vertices equal to $2\rho_D(S) +2$. This completes the proof.
\end{proof}

\section{Newton--Okounkov bodies with respect to $D$-positive flags}
\label{sec:positive-flags}

For the case of an \emph{ample} divisor $D$, \cite{RS} determined the possible numbers of vertices of Newton--Okounkov bodies of $D$ with respect to \emph{smooth} flags. 
In this section we extend this result to bodies of \emph{big and nef} divisors on smooth surfaces, with respect to smooth flags whose divisorial part intersects $D$ positively. 
Note that the latter hypothesis is automatically satisfied if $D$ is ample.
As a consequence, we can also deal with singular proper flags for ample divisors.

We begin by recalling the result of \cite{RS}. Let $S$ be a smooth projective surface.
For a configuration of curves $\mathscr{N}=(C_1,\dots,C_k)$ with negative definite intersection matrix, let $mc(\mathscr{N})$ denote the largest number of irreducible components of a connected subconfiguration of $\mathscr{N}$, and define
\[mv(\mathscr{N})=\begin{cases}
k+mc(\mathscr{N})+4 &\text{ if }k<\rho(S)-1,\\
k+mc(\mathscr{N})+3 &\text{ if }k=\rho(S)-1.
\end{cases}
\]
Then the maximum number of vertices of a Newton--Okounkov body of an ample divisor $D$ on $S$ is the maximum of all $mv(\mathscr{N})$ for $\mathscr{N}$ with negative definite intersection matrix \cite[Theorem 1.1]{RS}.

Now assume $D$ is merely big and nef rather than ample, and let $\Null(D)=\{E_1, \dots,E_\ell\}$. 
After the results of previous sections, it should not be a surprise that $\Null(D)$ needs to be taken into account to define a version of $mv$ appropriate for the present setting.
\begin{definition}\label{def:mv}
	For a configuration of curves $\mathscr{N}=(C_1,\dots,C_k,E_1, \dots,E_\ell)$ with negative definite intersection matrix that contains all curves of $\Null(D)$, let $mc_D(\mathscr{N})$ denote the largest number of irreducible components \emph{not in $\Null(D)$} which belong to a single connected component of $\mathscr{N}$. Note that the connected component may include irreducible components belonging to $\Null(D)$, but these do not contribute to $mc_D$. Then we define 
\[mv_D(\mathscr{N})=\begin{cases}
k+mc_D(\mathscr{N})+4 &\text{ if }k<\rho_D(S)-1,\\
k+mc_D(\mathscr{N})+3 &\text{ if }k=\rho_D(S)-1.
\end{cases}
\]
\end{definition}
We will prove below that the maximum number of vertices of a Newton--Okounkov body of $D$ with respect to a proper flag whose divisorial part intersects $D$ positively is 
\begin{equation}
\label{def:mvD}
mv(D):=\max \left\{ mv_D(\mathscr{N}) \,:\, \mathscr{N} \supseteq \Null(D), \text{negative definite}\right\}.
\end{equation}
It may be possible to obtain a Newton--Okounkov body $\Delta_{\Y}(D)$ with more than $mv(D)$ vertices using flags $\Y=\{S\supset C \supset \{p\}\}$ whose divisorial part $C$ belongs to $\Null(D)$, see Example \ref{exa:more-than-mv}.

If $D$ is a big but not necessarily nef divisor, we let $mv(D)=mv(P_D)$ where $P_D$ is the positive part in the Zariski decomposition $D=P_D+N_D$.

\begin{lemma}\label{lem:mv-birational}
	Let $S$ be a smooth projective surface, $D$ a big and nef divisor on $S$, and $\pi_p:\tilde{S}\rightarrow S$ be the blowing-up centered at a point $p\in S$.
	Then $mv(D)=mv(\pi_p^*(D))$.
\end{lemma}
\begin{proof}
First we observe that if $\Null(D)=\{E_1,\dots,E_\ell\}$, then we have $\Null(\pi_p^*(D))=\{\tilde{E}_1,\dots$, $\tilde{E}_\ell,E_p\}$, where for any curve $E$ on $S$ we will denote by $\tilde{E}$ its strict transform in $\tilde{S}$.
	
Second, a configuration of curves $\mathscr{N}=(C_1,\dots,C_k,E_1,\dots,E_\ell)$ has negative definite intersection matrix on $S$ if and only if $\tilde{\mathscr{N}}=(\tilde{C}_1,\dots,\tilde{C}_k,\tilde{E}_1,\dots$, $\tilde{E}_\ell,E_p)$ has negative definite intersection matrix on $\tilde{S}$. 
Indeed, there is an obvious equality of spans 
$$
\langle \tilde C_1,\dots,\tilde C_k,\tilde E_1,\dots, \tilde E_\ell, E_p\rangle=\langle \pi_p^*(C_1),\dots,\pi_p^*(C_k),\pi_p^*(E_1),\dots,\pi_p^*(E_\ell), E_p\rangle
$$
and, since $E_p$ is orthogonal to all pullbacks and has negative selfintersection, the intersection form on the span $\langle \pi_p^*(C_1),\dots$, $\pi_p^*(C_k),\pi_p^*(E_1),\dots,\pi_p^*(E_\ell), E_p\rangle$ is negative definite if and only if it is so on $\langle C_1,\dots,C_k$, $E_1,\dots,E_\ell\rangle$.
	
	Moreover, the connected components of $\mathscr{N}$ which contain curves not in $\Null(D)$ are in bijection with the connected components of $\tilde{\mathscr{N}}$ which contain curves not in $\Null(\pi_p^*(D))$. 
	
	Then, for every configuration $\mathscr{N}=(C_1,\dots,C_k,E_1,\dots,E_\ell)$ with $mv_D(\mathscr{N})=mv(D)$, the configuration of strict transforms $\tilde{\mathscr{N}}:=(\tilde{C}_1,\dots,\tilde{C}_k,\tilde{E}_1,\dots,\tilde{E}_\ell,E_p)$
	has $mv_{\pi_p^*(D)}(\tilde{\mathscr{N}})=mv_D(\mathscr{N})$; this shows that $mv(\pi_p^*(D))\ge mv(D)$.
	Finally, for every configuration $\tilde{\mathscr{N}}=(\tilde{C}_1,\dots,\tilde{C}_k,\tilde{E}_1,\dots,\tilde{E}_\ell,E_p)$ with $mv_{\pi_p^*(D)}(\tilde{\mathscr{N}})=mv(\pi_p^*(D))$, the configuration of push-forwards
	$\mathscr{N}=(C_1,\dots,C_k,E_1,\dots,E_\ell)$ has $mv_D(\mathscr{N})=mv_{\pi_p^*(D)}(\tilde{\mathscr{N}})$; this proves the inequality $mv(D)\ge mv(\pi_p^*(D))$, and the proof is complete.
\end{proof}

By Lemma \ref{lem:mv-birational}, $mv$ is a birational invariant, and this allows us to extend the definition to possibly singular surfaces. If $D$ is a big and nef divisor on a normal projective surface $S$, then we define $mv(D)$ to be equal to $mv(\pi^*(D))$ where $\pi:\tilde{S}\rightarrow S$ is a resolution of singularities of $S$. By the lemma and the fact that two arbitrary resolutions are dominated by some smooth model, $mv(\pi^*(D))$ does not depend on the chosen resolution $\pi$.

\begin{lemma}\label{lem:large-m-null}
	Let $B$ be a big and nef $\bbQ$-divisor on a smooth surface $S$.
	For every integer $m$ large and divisible enough, $C_m=mB-\Fix(mB)$ is a big and nef effective divisor with $\Null(C_m)\subset  \Null (B)$.
\end{lemma}
\begin{proof}
	Replacing $B$ by a suitable multiple, we may assume that $B$ is an effective divisor on $S$.
	Then  $\Fix(B)$ is an effective divisor $a_1F_1+\dots+a_kF_k$ such that for every positive integer $m$, the support of $\Fix(mB)$ is contained in $\Fix(B)$. Moreover, by Lemma \ref{lem:sb-nef}.(\ref{lem:sb-nef-bounded}), the coefficients of the components $F_i$ of $\Fix(B)$ in $\Fix(mB)$ are bounded by a constant $a$ independent of $m$, i.e., $\Fix(mB)\le a(F_1+\dots+F_k)$ for every $m$. As the big cone is open, for $m$ large enough the divisor
	\[Z_m=mB-a(F_1+\dots+F_k)\]
	is big and effective, so $C_m\ge Z_m$ is big and effective as well.
	On the other hand, $\Fix(C_m)=0$ by definition of $C_m$, so $C_m$ is nef.
	
	Fix now $m_0$ such that $Z_{m_0}$ is big and effective, and let $Z_{m_0}=P+N$ be its Zariski decomposition.
	For every $m$, $E_m=C_m-Z_m$ is effective and supported on $\Fix(B)$, so if $m\ge m_0$, we can write
	$$C_m=E_m+Z_m=E_m+P+N+(m-m_0)B\, .$$
	For every curve $E$ which is not a component of $\Fix(B)$ or $\Null(P)$,
	one has $E_m\cdot E\ge 0$, $N\cdot E\ge 0$ (because every component of $N$ belongs to $\Null(P)$), $B \cdot E \ge 0$ and $P\cdot E>0$, so $C_m\cdot E>0$. It follows that every curve in $\Null(C_m)$ for $m$ large enough is either a component of $\Fix(B)$ or $\Null(P)$, which is a finite collection of curves.
	
	Now to show that $\Null(C_m)\subset  \Null (B)$ for $m$ large enough, observe that for every curve $E$ not in $\Null(B)$,
	\[C_m\cdot E=mB\cdot E -\Fix(mB)\cdot E \ge m(B\cdot E) -a \left(\sum_{i=1}^{k} F_i\cdot E\right) \]
	is positive for $m$ larger than some integer $m_E$. It follows that the claim is satisfied by every integer $m$ larger than $m_0$ and the finitely many numbers $m_E$ where $E$ is a component of $\Fix(B)$ or $\Null(P)$ not in $\Null(B)$.
\end{proof}

\begin{lemma}\label{ordered-negative}
	Let $S$ be a smooth projective surface, and $D$ a big and nef divisor on $S$ with $\Null(D)=\{E_1,\dots,E_\ell\}$. 
	Let  $C_1,\dots,C_k$ be irreducible curves on $S$ such that the intersection matrix of $\mathscr{N}=(C_1,\dots,C_k,E_1,\dots,E_\ell)$ is negative definite.
	Then there is an irreducible curve $C$ with $\Null(C)\subseteq\Null(D)$ with the following properties (using Notation \ref{notation}).
	\begin{enumerate}
		\item $C$ intersects each $C_i$, $i=1,\dots, k$, in at least two points.
		\item For every $t$ with $D_t=D-tC$ pseudo-effective, the negative part $N_t$ of its Zariski decomposition is supported on $\mathscr{N}$.
		\item \label{allC} Every $C_i$, $i=1,\dots, k$, is a component of $N_\mu$.
		\item \label{orderC} $t_{C_1}< \dots < t_{C_k}$.
	\end{enumerate}
	Moreover, $C$ can be chosen in the linear span of $D$, the $C_i$ and the $E_j$ and satisfying $C\cdot D>0$.
\end{lemma}

\begin{proof}
	This proof is adapted to the present setting from \cite[Lemma 5.3]{RS}.
	
	We will prove by induction on $k$ that there are positive rational numbers $a_1, \dots, a_k$ such that, denoting
	$$D-a_1C_1-\dots-a_kC_k=P+N$$ 
	the Zariski decomposition of the divisor on the left hand side, then $P$
	is big and nef with $\Null(P)\subseteq\Null(D)$, and for every sufficiently large $m$ and every irreducible curve $C\in|mP-\Fix(mP)|$ properties \eqref{allC} and \eqref{orderC} are satisfied.
	By Lemma \ref{lem:sb-nef} and Bertini's theorem, for $m$ large and divisible enough there are irreducible curves in $|mP-\Fix(mP)|$ that intersect $D$ positively and each $C_i$ in at least two points, so by Lemma \ref{lem:large-m-null} we shall be done.
	
	If $k=1$, choose a positive integer $a$ such that the divisor class $B=D-(1/a)C_1$ is big and its Zariski decompostion $B=N+P$ satisfies $\Null(P)=\Null(D)$. This is certainly possible: on one hand, because the big cone is open, $B$ is big for $1/a$ small enough; on the other hand, by the continuity of Zariski decomposition and \cite[Proposition 1.3]{BKS}, for $1/a$ small enough we have $N \cup \Null(P)\subseteq\Null(D)$, and every component $E_j$ of $\Null(D)$ which does not satisfy $C_1 \cdot E_j=0$ belongs to $N$, so $N\cup \Null(P)\supseteq\Null(D)$.
	
	Then for every $m$ and every $C\in |mP-\Fix(mP)|$, 
	$$D_{1/m}=D-(1/m)C=(1/a)C_1+N+(1/m)\Fix(mP)$$ 
	is exactly equal to $N_{1/m}$, as it is an effective divisor with negative definite intersection matrix because $\supp(N+(1/m)\Fix(mP))\subseteq\Null(mP)=\Null(D)$ since $P$ is nef (this follows from \cite[Theorem 8.1]{Zar62}). Therefore	$0<t_{C_1}<1/m$,
	and we are done.
	
	Now assume the claim is true for $\mathscr{N}'=(C_1,\dots,C_{k-1},E_1,\dots,E_\ell)$, and let $a,a_1,\dots,a_{k-1}$ be positive rational numbers such that the Zariski decomposition $$B'=D-a_1C_1-\dots-a_{k-1}C_{k-1}=P'+N'$$ 
	has $P'$ big with $\Null(P')=\Null(D)$ and moreover
	\begin{enumerate}
		\item if we denote by $N_t'$ the negative part of the Zariski decomposition of $D'_t=D-tP'$, then
		$\sup\{t\in \bbQ \,:\, C_i \text{ is not contained in }N_t'\}$ is a finite positive real number $t_{C_i}'$,  and
		\item $t_{C_1}'< \dots < t_{C_{k-1}}'$.
	\end{enumerate}
	In this case, since $P'$ is big and nef, Wilson's theorem \cite[Theorem 2.3.9]{La1} guarantees that for $m\gg0$ and every $C'\in |mP'-\Fix(mP')|$, the numbers $t_{C_i}$ such that $C_i$ is a component of $N_t$ for $t>t_{C_i}$ satisfy $t_{C_1}<\dots<t_{C_{k-1}}<1$. 
	
	Moreover, for every $t\in[0,1]$, we have $D_t=(1-t)D+t(ma_1C_1+\dots+ma_{k-1}C_{k-1}+mN'+\Fix(mP'))$, with $(1-t)D$ nef and $ma_1C_1+\dots+ma_{k-1}C_{k-1}+\Fix(mP')$ effective, so by the extremality properties of the Zariski decomposition it follows that 
	$$
	N_t\le t(ma_1C_1+\dots+ma_{k-1}C_{k-1}+N'+\Fix(mP'))
	$$
	 (with equality if and only if $t=1$). In particular, all components of $N_t$ are among the $C_i$ or the $E_j$.
	
	Choose rational numbers $s_i$ with $0=s_0<t_1'<s_1<t_2'<\dots<s_{k-2}<t_{k-1}'<s_{k-1}<1$.
	The choices made guarantee that the irreducible components of $N'_{s_i}$ not in $\Null(D)$ are exactly $C_1, \dots, C_i$, and $P'_{s_i}\cdot C_j>0$ for all $i<j\le k$.
	Therefore by continuity of the Zariski decomposition (\ref{pro:continuity-zariski}), there exist $\varepsilon_1,\dots,\varepsilon_{k}>0$ such that choosing $a_k\in \bbQ$ with $a_k\le \varepsilon_i$, the irreducible components not in $\Null(D)$ of the negative part in the Zariski decomposition of $D-s_i(P'-a_kC_k)$	are also exactly $C_1, \dots, C_i$.
	Thus by choosing a rational number $a_k$ smaller than $\varepsilon_0, \dots, \varepsilon_{k-1}$ and setting 
	\begin{gather*}
		D-a_1C_1-\dots-a_kC_k=P+N\ ,\\
		C\in|mP-\Fix(mP)|,
	\end{gather*}
	again Wilson's theorem guarantees that for $m\gg0$ the irreducible components  not in $\Null(D)$ of the negative part in the Zariski decomposition $P_{s_i}+N_{s_i}$ of $D-s_iC$ are $C_1, \dots, C_i$.
	On the other hand, in the Zariski decomposition
	$D-C=P_1+N_1$
	one has $N_1=ma_1C_1+\dots+ma_kC_k+N+\Fix(mP)$ and therefore
	$$t_{C_{k-1}}<s_{k-1}<t_{C_k}<1,$$
	which completes the induction step.
	
	For the last claim, we first observe that the class $P$ thus constructed is a combination of $D$, the $C_i$ and the $E_j$, and hence (since $\Null(P)=\Null(D)$) the class of $C'$ is also a combination of $D$, the $C_i$ and the $E_j$. 
	Moreover, $C'\cdot D\ge 0$, and we can slightly modify $P$ to obtain a $P''$ to guarantee that $C''\cdot D>0$ (and still satisfy the properties) by taking $P''=P+\varepsilon D$ for $\varepsilon$ small enough.
\end{proof}

\begin{lemma}\label{ordered-negative-maximal}
	Let $D$ a big and nef divisor on $S$ with $\Null(D)=\{E_1,\dots,E_\ell\}$.
	Assume $k<\rho_D(S)-1$ and $C_1,\dots,C_k$ are irreducible curves on $S$ such  $\mathscr{N}=(C_1,\dots,C_k,E_1,\dots,E_\ell)$ is a maximal effective divisor with negative definite intersection matrix, i.e., such that there exists no curve $C'$ distinct from $C_1,\dots,C_k,E_1,\dots,E_\ell$ with $\mathscr{N}+C'$ having negative definite intersection matrix.
	Then the irreducible curve $C$ from Lemma \ref{ordered-negative} can be assumed to have a numerical class linearly independent from $\langle D, E_1,\dots,E_\ell,C_1, \dots, C_k\rangle$ in $\NS(S)_\bbR$.
\end{lemma}

\begin{proof}
	The argument from \cite[Lemma 5.4]{RS} applies verbatim, with $B=P$.
\end{proof}

\begin{theorem}\label{thm:existence-positive}
	On every smooth projective surface $S$, for every big and nef divisor $D$ and every integer $\delta$, $3\le \delta \le mv(D)$, there exists a smooth flag $\Y=\{S\supset C \supset \{p\}\}$ with $D\cdot C>0$ such that the Newton--Okounkov polygon \(\Delta_{\Y}(D)\) has exactly \(\delta\) vertices.
\end{theorem}
\begin{remark}\label{rem:mv-determined-null}
	By its definition in \eqref{def:mvD}, the number $mv(D)$ only depends on the set $\Null(D)$, i.e., all big and nef divisors $D$ with the same set $\Null(D)$ have Newton--Okounkov bodies with $3, \dots, mv(D)$ vertices.
\end{remark}

\begin{proof}
	Write as above $\Null(D)=\{E_1,\dots,E_\ell\}$.
	Choose a configuration $\mathscr{N}_0=(C_1,\dots,C_k$, $E_1,\dots,E_\ell)$ with negative definite intersection matrix such that $mv(D)=mv_D(\mathscr{N}_0)$. 
	Assume moreover that its components have been ordered in such a way that one connected component of $\mathscr{N}_0$ contains all $C_i$ for $1\le i \le mc(\mathscr{N}_0)$.
		By the definition of $mv(\mathscr{N})$, it is not restrictive to assume that $\mathscr{N}_{0}$ is maximal, i.e., there exists no curve $C'$ with $\mathscr{N}_{0}+C'$ having negative definite intersection matrix.

	Then, if $k<\rho_D(S)-1$, for every $mc_D(\mathscr{N}_0)\le i\le k$, $mv_D(C_1,\dots,C_i,E_1,\dots,E_\ell)=mv_D(\mathscr{N}_0)-k+i$, and for $0\le i \le mc_D(\mathscr{N}_0)$, $mv_D(C_1,\dots,C_i,E_1,\dots,E_\ell)=mv_D(\mathscr{N}_0)-k-mc_D(\mathscr{N}_0)+2i$.
	On the other hand, if $k=\rho_D(S)-1$, then
	
	\[mv_D(C_1+\dots+C_i)=\begin{cases}
mv_D(\mathscr{N}_0)-k+i+1 &\text{ for every }mc_D(\mathscr{N}_0)\le i< k,\\
mv_D(\mathscr{N}_0)-k-mc_D(\mathscr{N}_0)+2i+1 &\text{ for  }0\le i \le mc_D(\mathscr{N}_0).
\end{cases}
\]
In any event,
	\[\{3,\dots,mv(D)\}=\{mv_D(\mathscr{N}) : \mathscr{N}\le \mathscr{N}_0\}\cup\{mv_D(\mathscr{N})-1: \mathscr{N}\le \mathscr{N}_0\}. \]
	Therefore, it will be enough to prove that, for every choice of irreducible curves $C_1,\dots,C_k$ on $S$ such that the intersection matrix of $\mathscr{N}=(C_1,\dots,C_k,E_1,\dots,E_\ell)$ is negative definite, 
	\begin{itemize}
	\item If $\mathscr{N}$ is maximal, there is a flag $\Y$ such that $\Delta_{\Y}(D)$ has $mv_D(\mathscr{N})$ vertices.
	\item If $\mathscr{N}$ is nonzero or has less than $\rho_D(S)-1$ components, there is a flag $\Y$ such that $\Delta_{\Y}(D)$ has $mv_D(\mathscr{N})-1$ vertices.
	\item If $\mathscr{N}$ is nonzero and has less than $\rho_D(S)-1$ components, there is a flag $\Y$ such that $\Delta_{\Y}(D)$ has $mv_D(\mathscr{N})-2$ vertices.	
\end{itemize}
	Moreover, in each case the divisorial of $\Y$ can be assumed to intersect $D$ positively.
	
	In the case of a maximal $\mathscr{N}$ with less than $\rho_D(S)-1$ components, choose an irreducible curve $C$ satisfying the conditions of Lemma \ref{ordered-negative-maximal}, and let $p$ be one of the intersection points of $C$ and $C_1$ (unless $\mathscr{N}=0$ in which case we choose an arbitrary $p\in C$).
We claim that $D$, $\Y:S\supset C \supset\{p\}$ give a body with $mv_D(\mathscr{N})$ vertices.
On the one hand, since $C\cdot D>0$, it follows that $P_0\cdot C>0$, so $\nu=0$ and $\Delta_{\Y}(D)$ has two leftmost vertices.
Moreover, Theorem \ref{pro:KLM} ensures that $\Delta_{\Y}(D)$ has two interior vertices with first coordinate equal to the number $t_i$ given by Lemmas \ref{ordered-negative} and \ref{ordered-negative-maximal} for $i=1,\dots,mc_D(\mathscr{N})$, whereas it only has an upper interior vertex for $mc_D(\mathscr{N})<i\le k$.
Finally, as the numerical class of $C$ is independent from $\langle D, E_1,\dots,E_\ell,C_1, \dots, C_k\rangle$ in $\NS(S)_\bbR$, by Corollary \ref{cor:vertices} $\Delta_{\Y}(D)$ has two rightmost vertices.
So, the total number of vertices is $mv_D(\mathscr{N})$.

Now choose $C$ verifying the conditions of Lemma \ref{ordered-negative}, so that the class of $C$ belongs to the span of $D$, the $C_i$ and the $E_j$ . The shape of $\Delta_{\Y}(D)$ is as before, but with a single rightmost vertex; if $\mathscr{N}$ has $\rho_D(S)-1$ components (in particular $\mathscr{N}$ is maximal) the total number of vertices is $mv_D(\mathscr{N})$, otherwise it is $mv_D(\mathscr{N})-1$.

Finally, if $\mathscr{N}$ is nonzero we can pick $p$ differently, while keeping the same curve $C$ that satisfies the condtions of Lemma \ref{ordered-negative}. If $mc_D(\mathscr{N})=1$ we let $p$ be a point of $C$ not on $\mathscr{N}$, and if $mc_D(\mathscr{N})>1$ then we take $p$ to be one of the intersection points of $C$ with $C_2$. In this way we obtain one lower point less, so if $\mathscr{N}$ has $\rho_D(S)-1$ components the total number of vertices is $mv_D(\mathscr{N})-1$, otherwise it is $mv_D(\mathscr{N})-2$.
	\end{proof}

\begin{corollary}\label{cor:vertices-proper-flag}
	Let $D$ be a big and nef divisor on the normal projective surface $S$.
	Then there is a proper flag $\Y=\{\tilde{S}\supset C \supset \{p\}\}$ over $S$ such that $D\cdot \pi(C)>0$ and $\Delta_{\Y}$ is a $k$-gon if and only if $3\le k \le mv(D)$. 
\end{corollary}
	Recall that $\Y$ is a proper flag over $S$ if $\pi:\tilde{S}\to S$ is a proper birational morphism, $p$ is a smooth point of $C$ and $\tilde{S}$, and $\pi(C)$ is a curve on $S$. Note that the corollary describes all Newton--Okounkov bodies of ample divisors with respect to proper flags. 
\begin{proof}
	By the definition of $mv(D)$, we may replace $S$ by a resolution of singularities and assume $S$ is smooth.
	
	The existence part is then given by Theorem \ref{thm:existence-positive}, so we only need to show that for every flag $\Y=\{\tilde{S}\supset C \supset \{p\}\}$ where $\pi:\tilde S \rightarrow S$ is a proper birational map between smooth surfaces, such that $D\cdot \pi(C)>0$, the number of vertices of $\Delta_{\Y}(D)$ is bounded above by $mv(D)$.
	
	By Lemma \ref{lem:mv-birational}, it holds that $mv(D)=mv(\pi^*(D))$, and by the projection formula \cite[8.1.7]{Ful98}, $\pi^*(D)\cdot C=D\cdot \pi(C)>0$, whereas $\Delta_{\Y}(D)=\Delta_{\Y}(\pi^*(D))$ as already observed in the proof of Theorem \ref{thm:bound}. Therefore we may replace $S$ with $\tilde{S}$ and it is enough to prove the claim for smooth flags $\Y=\{S\supset C\supset \{p\}\}$.
	
	Keeping Notation \ref{notation}, let as above 
	$\Null(D)=\{E_1,\dots, E_\ell\}$, and let $\{C_1, \dots, C_k\}$ be the components of $N_\mu$ not in $\Null(D)$.
	Choose a number $\mu'<\mu$ such that the support of $N_{\mu'}$ is the same as the support of $N_\mu$.
	Since $P_{\mu'}=D-{\mu'} C -N_{\mu'} \le D$ and by hypothesis $D\cdot C>0$, for every component $E_i$ of $\Null(D)$ which is not a component of $N_{\mu'}$ we have $P_{\mu'} \cdot E_i= D\cdot E_i- {\mu'} C \cdot E_i- N \cdot E_i\le 0$, which, $P_{\mu'}$ being nef, implies $P_{\mu'} \cdot E_i=0$. 
	Therefore all curves in the configuration $\mathscr{N}=(C_1,\dots, C_k, E_1,\dots, E_\ell)$ have intersection number 0 with $P_{\mu'}$, which is big (because $\mu'<\mu$) so by the Index Theorem $\mathscr{N}$ has negative definite intersection matrix.
	
	Then by \cite[Proposition 3.3]{RS} and Lemma \ref{lem:pullback}, the number of vertices of $\Delta_{\Y}(D)$ is at most $mv_D(\mathscr{N})$. 
\end{proof}


\section{Newton--Okounkov bodies with respect to $D$-orthogonal flags}
\label{sec:examples}

It is a subtler problem to determine the number of vertices of the Newton--Okounkov bodies $\Delta_{\Y}(D)$ where the divisorial part of the flag $\Y$ belongs to $\Null(D)$.
To begin with, there are only finitely many curves in $\Null(D)$, which entails that there are only finitely many such bodies, and the facts explained in Theorem \ref{thm:existence-positive} and Remark \ref{rem:mv-determined-null} do not hold anymore.

\begin{example}\label{ex:51}
	There exist surfaces $S$ with divisors $D, D'$ satisfying $\Null(D)=\Null(D')$ for which the numbers $k$ such that some $k$-gon arises as Newton--Okounkov body of $D$ or $D'$ with respect to flags with divisorial part in $\Null(D)$ differ.
	
	Let $S$ be the blowup of $\bbP^1 \times \bbP^1$ at a point $p$. There are three curves of negative selfintersection on $S$, namely the exceptional divisor $E_p$ and the strict transforms $F_1$, $F_2$ of the two rulings passing through $p$. Consider the divisors $D=2E_p+F_1+F_2$ and $D'=3E_p+2F_1+F_2$. It is immediate that $\Null(D)=\Null(D')=\{E_p\}$. In addition, consider flags of the form $\Y=\{S\supset E_p \supset \{q\}\}$. According to Galindo, Monserrat and Moreno-\'Avila \cite[Theorem 3.6]{GMMA}, the first component $v_1$ of the valuation $v_{\Y} = (v_1,v_2)$ is a non-positive at infinity divisorial valuation, hence we can compute the Newton--Okounkov bodies of $D$ resp. $D'$ with respect to flags $\Y$ using their explicit calculations presented in \cite{GMMA2}. More precisely, we distinguish three cases:
\begin{enumerate}
\item If $q \in E_p \cap F_1$, then \cite[Theorem 3.12]{GMMA2} applies so that both $\Delta_{\Y}(D)$ and $\Delta_{\Y}(D')$ are quadrilaterals.	
\item If $q \in E_p \cap F_2$, then \cite[Theorem 3.13 (a)]{GMMA2} implies again that both $\Delta_{\Y}(D)$ and $\Delta_{\Y}(D')$ are quadrilaterals.
\item Finally, if $q\in E_p$  does not belong neither to $E_p\cap F_1$ nor to $E_p\cap F_2$, then the Newton--Okounkov body $\Delta_{\Y}(D)$ is a triangle whereas $\Delta_{\Y}(D')$ is a quadrilateral by \cite[Theorem 3.12]{GMMA2}.

\begin{figure}[h]
     \centering
\begin{subfigure}[b]{0.4\textwidth}
         \centering
         \begin{tikzpicture}
        \tkzInit[xmax=3,ymax=2,xmin=-0.5,ymin=-0.5]
        \tkzGrid
        \tkzAxeXY
        \draw[thick, pattern=crosshatch dots, pattern color=black] (0,0) -- (1,1) -- (2,1) -- (1,0) -- (0,0);
        \filldraw[black] (0,0) circle (2pt);
        \filldraw[black] (1,1) circle (2pt);
        \filldraw[black] (2,1) circle (2pt);
        \filldraw[black] (1,0) circle (2pt);
        \end{tikzpicture}
         \caption{$\Delta_{\Y}(D)$ for $q\in E_p\cap F_1$.}
         \label{case1}
\end{subfigure}
\begin{subfigure}[b]{0.4\textwidth}
         \centering
          \begin{tikzpicture}
          \tkzInit[xmax=3,ymax=2,xmin=-0.5,ymin=-0.5]
          \tkzGrid
          \tkzAxeXY
           \draw[thick, pattern=crosshatch dots, pattern color=black] (0,0) -- (1,1) -- (2,2) -- (3,2) -- (2,1) -- (1,0) -- (0,0);
        \filldraw[black] (0,0) circle (2pt);
        \filldraw[black] (2,2) circle (2pt);
        \filldraw[black] (3,2) circle (2pt);
        \filldraw[black] (1,0) circle (2pt);
          \end{tikzpicture}
         \caption{$\Delta_{\Y}(D')$ for $q\in E_p\cap F_1$.}
         \label{case1p}
\end{subfigure}

\begin{subfigure}[b]{0.4\textwidth}
         \centering
         \begin{tikzpicture}
        \tkzInit[xmax=3,ymax=2,xmin=-0.5,ymin=-0.5]
        \tkzGrid
        \tkzAxeXY
        \draw[thick, pattern=crosshatch dots, pattern color=black] (0,0) -- (1,1) -- (2,0) -- (0,0);
        \filldraw[black] (0,0) circle (2pt);
        \filldraw[black] (1,1) circle (2pt);
        \filldraw[black] (2,0) circle (2pt);
        \end{tikzpicture}
         \caption{$\Delta_{\Y}(D)$ for $q\notin F_1, F_2$.}
         \label{case2}
\end{subfigure}
\begin{subfigure}[b]{0.4\textwidth}
         \centering
          \begin{tikzpicture}
          \tkzInit[xmax=3,ymax=2,xmin=-0.5,ymin=-0.5]
          \tkzGrid
          \tkzAxeXY
           \draw[thick, pattern=crosshatch dots, pattern color=black] (0,0) -- (1,1) -- (2,1) -- (3,0) -- (0,0);
           \filldraw[black] (0,0) circle (2pt);
           \filldraw[black] (1,1) circle (2pt);
           \filldraw[black] (2,1) circle (2pt);
           \filldraw[black] (3,0) circle (2pt);
          \end{tikzpicture}
         \caption{$\Delta_{\Y}(D')$ for $q\notin F_1, F_2$.}
         \label{case2p}
\end{subfigure}
        \caption{Newton--Okounkov bodies in Example \ref{ex:51}}
        \label{fig:nobodies}
\end{figure}
\end{enumerate}

The cases (1) and (3) are illustrated in Figure \ref{fig:nobodies}, whereas case (2) behaves in the same manner as (1). These results agree Theorem \ref{thm:existence-positive}.

\end{example}

\begin{example}\label{exa:more-than-mv}
	It is also possible to obtain Newton--Okounkov bodies of $D$ with more than $mv(D)$ vertices by using flags whose divisorial part is orthogonal to $D$.
	
	Let $S\rightarrow\bbP^2$ be a proper birational map (which can be factored as a composition of  point blowups), and let $D$ be the pullback to $S$ of the class of a line. Obviously $\Null(D)$ consists of all the exceptional components, so $\rho_D(S)=1$ and only $\mathscr{N}=\Null(D)$ satisfies the conditions of Definition \ref{def:mv}. Therefore $mv(D)=mv_D(\mathscr{N})=3$. However, it was shown in \cite{CFKLRS} and \cite{GMMFN} that there are choices of $S$ that support flags $\Y$ such that $\Delta_{\Y}(D)$ is a quadrilateral.
\end{example}

	Note that in Example \ref{exa:more-than-mv}, four vertices is the maximal number $2\rho_D(S)+2$ allowed by Theorem \ref{thm:bound}. However, for many surfaces $S$ and divisors $D$ (e.g., the line on $\bbP^2$, but also its pullback to most of the blowups considered in \cite{CFKLRS} and \cite{GMMFN}) there is no flag $\Y$ such that $\Delta_{\Y}(D)$ has $2\rho_D(S)+2$ vertices.
	A general upper bound which takes into account the geometry on $S$ can be given following the ideas of the previous section using the following definition:
\begin{definition}\label{def:mv-orthogonal}
	Assume $D$ is a big and nef divisor such that $\Null(D)=\{E_1,\dots,E_\ell\}$ is nonempty.
	For a configuration of curves $\mathscr{N}=(C_1,\dots,C_k,E_{i_1}, \dots,E_{i_{\ell-1}})$ with negative definite intersection matrix that contains all curves of $\Null(D)$ \emph{but one}, let $mc_D(\mathscr{N})$ denote the largest number of irreducible components \emph{not in $\Null(D)$} which belong to a single connected component of $\mathscr{N}$. Then we define 
	\[mv_D(\mathscr{N})=\begin{cases}
	k+mc_D(\mathscr{N})+3 &\text{ if }k<\rho_D(S),\\
	k+mc_D(\mathscr{N})+2 &\text{ if }k=\rho_D(S),
	\end{cases}
	\]
	and 
	$$
	mv^{\Null}(D):=\max\{mv_D(\mathscr{N})\,:\, \mathscr{N} \text{ negative definite} ,\  \#(\Null(D)\setminus \mathscr{N})=1 \}.
	$$
\end{definition}

\begin{proposition}\label{pro:upper-bound-orthogonal}
On every smooth projective surface $S$, for every big and nef divisor $D$, and every admissible flag $\Y$ on $S$ whose divisorial part belongs to $\Null(D)$, the polygon $\Delta_{\Y}(D)$ has at most $mv^{\Null}(D)$ vertices.
\end{proposition}

The proof of the upper bound in Corollary \ref{cor:vertices-proper-flag} carries over to this setting, and we omit it for brevity. 
On the other hand, the analogous result to the existence part (i.e., Theorem \ref{thm:existence-positive}) does not hold for $D$-orthogonal flags.


Invariance under pullback by birational morphisms (i.e., Lemma \ref{lem:mv-birational}) also fails for $mv^{\Null}$. However, this added complexity will turn out to be useful in the next section.

\begin{example}
	Let $\pi:S\rightarrow\bbP^2$ be the blowup centered at a point $p$, and let $D$ the pullback of a line by $\pi$. Then $\Null(D)$ consists of the exceptional divisor $E_p$ only, so $mv^{\Null}(D)=3$. However, there exist further blowups $\eta:\tilde{S}\rightarrow S$ that allow flags $\Y$ such that $\Delta_{\Y}(\eta^*(D))$ is a quadrilateral, cf~Example \ref{exa:more-than-mv}. 
	Therefore, Proposition \ref{pro:upper-bound-orthogonal} implies
	$$mv^{\Null}(\eta^*(D))\ge 4 > mv^{\Null}(D).$$
\end{example}


\section{Infinitesimal flags with maximal negative configurations}
\label{sec:existence}

In the same way that, in Section \ref{sec:positive-flags}, results on Newton--Okounkov bodies with respect to $D$-positive admissible flags were used to describe the Newton--Okounkov bodies with respect to proper (non-necessarily admissible) flags, we next use our results on Newton--Okounkov bodies with respect to $D$-orthogonal flags to describe Newton--Okounkov bodies with respect to infinitesimal flags. 
However, working with infinitesimal flags provides additional flexibility (replacing $S$ with a suitable blowup) which will turn out to be enough to characterize surfaces of Picard number one in terms of Newton--Okounkov bodies.

It will simplify the presentation to use a notion of \emph{relative Zariski decomposition} inspired in \cite[Section 8]{CS93}.
\begin{definition}
	Let $D$ be an arbitrary divisor on the smooth projective surface $S$, and let $\mathscr{N}=\{E_1, \dots, E_k\}$ be a configuration of (irreducible, reduced) curves with negative definite intersection matrix. There is a unique effective $\bbQ$-divisor $N\subn=a_1E_1+\dots+a_kE_k$ 
	satisfying
	\begin{enumerate}
		\item $(D-N\subn)\cdot E_i\ge 0$ for all $i=1,\dots,k$.
		\item $(D-N\subn)\cdot E_i= 0$ for every $i$ such that $a_i>0$.
	\end{enumerate}
	We then set $P\subn=D-N\subn$ and call $D=P\subn+N\subn$ the Zariski decomposition of $D$ with respect to $\mathscr{N}$.
\end{definition}

The properties of the relative Zariski decomposition that we will need are contained in the following lemmas:

\begin{lemma}\label{lem:relzariski-vs-zariski-decomp}
	If $D$ is pseudo-effective and $D=N+P$ is its Zariski decomposition, then for every negative definite configuration $\mathscr{N}$, the relative Zariski decomposition $D=N\subn+P\subn$ satisfies $N\subn\le N$.
\end{lemma}

\begin{lemma}\label{lem:relzariski-p-positive-then-big}
	Let $D=N\subn+P\subn$ be the relative Zariski decomposition of a divisor $D$ with respect to some negative definite configuration $\mathscr{N}$.  If $P\subn^2> 0$, then $D$ is big. 
\end{lemma}

Both lemmas are elementary and we leave their proofs to the reader.
\medskip

Fix a big and nef divisor $D$ on a smooth projective surface $S$, with relative Picard number $\rho_D(S)=\rho$.
Our construction starts from an irreducible nodal curve $C$ on $S$, which for simplicity we assume has positive selfintersection, and we will show that there is an infinitesimal flag $\Y$ centered at the node $p$ of the curve $C$, such that $C$ induces two internal vertices on the Newton--Okounkov body of $D$ with respect to $\Y$. This allows us to prove that $D$ admits Newton--Okounkov bodies with four vertices, and also with five if $\rho>1$.

For each \(k>0\), let us denote by \(S_k \rightarrow S\) the composition of the \(k\) successive
blowups
\[S_k \rightarrow S_{k-1} \rightarrow \dots \rightarrow S_1 \rightarrow S_0=S,\]
where \(S_1\rightarrow S_0\) is the blowup with center at \(p\),
\(S_2 \rightarrow S_1\) is the blowup with center at one of the two points
in \(C_1 \cap E_1\), and for every \(i>1\), \(S_{i+1}\rightarrow S_{i}\)
is the blowup with center at the point \(p_i=C_{i}\cap E_{i}\), where \(C_i\) denotes the strict transform of \(C\) on \(S_i\), and $E_i$ stands for the exceptional divisor of $S_i\rightarrow S_{i-1}$. 

Slightly abusing notations, for every divisor $Z$ on a surface $X$ we shall denote by the same symbol $Z$ its pullback by a blowing-up morphism (the spirit of this abuse of notation is consistent with a birational point of view which thinks of $Z$ as a $b$-divisor, see \cite{Cor07}).

Denoting by \(E_{i,k}\) the strict transform on \(S_k\) of the exceptional divisor \(E_i\), the pullback of \(C\) on \(S_k\) (which ---as said above--- we continue to write \(C\) for simplicity of notation)
is \(C_k+2E_{1,k}+3E_{2,k}+\dots+(k+1)E_{k,k}\) and the pullback of
\(E_{i}\) (written \(E_i\) for simplicity) is \(E_{i,k}+E_{i+1,k}+\dots+E_{k,k}\).

All divisors on \(S_k\) we shall be concerned with belong to numerical equivalence classes which are linear combinations of \([D], [C], [E_1], \dots, [E_k]\). 
They are moreover orthogonal to \(E_{i,k}\) for \(i=2,\dots,k-1\) and thus they actually
belong to the span \(\langle [D], [C], [E], [E_1]\rangle\) where
\(E=\sum_{i=1}^{k} i E_{i,k}\). In particular, \(C_k=C-E-E_1\). Note also that \(C_k\) intersects \(E\) at the components \(E_{1,k}\) and $E_{k,k}$ of $E$.

\begin{proposition}\label{pro:nodal-curve-two-internal}
	Let $D$ be a big and nef divisor on a smooth projective surface $S$, and $C$ an irreducible curve on $S$ with a node at the point $p$ and not belonging to $\Null(D)$. 
	For every integer $k>0$ let $S_k\rightarrow S$ be the composition of blowups determined by $C$ as described above, and consider $\Y=\{S_k\supset E_k\supset \{p_k\}\}$ where $p_k=C_k\cap E_k$.
	If \(k>(D\cdot C)/\sqrt{D^2}-1\), then $\Delta_{\Y}(D)$ has at least two internal vertices.
\end{proposition}

Note that choosing $p_k'=E_k\cap E_{k-1,k}$ would lead to the same conclusion.

\begin{proof}
The determination of $\Delta_{\Y}(D)$ relies on the understanding of the Zariski decomposition of \(D_t=D-t E_{k}\), which we denote by \(D_t=N_t+P_t\).

Clearly, the divisors $E_{i,k}$ for $i=1,\dots, k-1$ belong to $N_t$ as soon as $t>0$. In fact, for $\mathscr{N}=\{E_1,\dots, E_{k-1}\}$, it is easy to check that the relative Zariski decomposition $D_t=P_{t,\mathscr{N}}+N_{t,\mathscr{N}}$ of $D_t$ with respect to $\mathscr{N}$ has negative part $N_{t,\mathscr{N}}=\frac{t}{k}E - tE_k$, and we know (by Lemma \ref{lem:relzariski-vs-zariski-decomp}) that $N_t\ge N_{t,\mathscr{N}}$. 
Therefore, for every \(t>0\) such that \(D_t\) is pseudo-effective, either \(C_k\) is a component of \(N_t\) or $C_k\cdot P_{t,\mathscr{N}}$ is nonnegative, which gives
\[0\le C_k\cdot \left( D_t-\left(\frac{t}{k}E - tE_k\right)\right)=(C-E-E_1)\cdot (D-\frac{t}{k}E)=D\cdot C - t\frac{k+1}{k}\ .\]
In other words, for $t>\frac{kD\cdot C}{k+1}$ such that $D_t$ is pseudoeffective, $C_k$ belongs to $N_t$. 

On the other hand,  by Lemma \ref{lem:relzariski-p-positive-then-big} we know that $D_t$ is big whenever 
$$
0<P_{t,\mathscr{N}}^2=\Big( D_t-\Big(\frac{t}{k}E - tE_k\Big)\Big)^2=\Big (D-\frac{t}{k}E\Big )^2=D^2-t^2/k.
$$ 
The hypothesis $k>(D\cdot C)/\sqrt{D^2}-1$ guarantees the existence of a $t>\frac{kD\cdot C}{k+1}$ such that $D^2-t^2/k>0$, i.e., with $D_t$ big and therefore pseudoeffective. 
Hence for $t$ large enough, $C_k$ does belong to $N_t$, whereas for $t$ close to zero $C_k$ does not belong to $N_t$, because $D_0\cdot C_k=D\cdot C_k>0$.
Thus there is a value $t_{C_k}$ (less than or equal to $\frac{kD\cdot C}{k+1}$) such that $C_k$ belongs to $N_t$ if and only if $t>t_{C_k}$, hence there are interior vertices on $\Delta_{\Y}(D)$ with abscissa $t_{C_k}$. 
Since $C_k$ meets the configuration $\mathscr{N}$ at the component $E_{1,k}$, $N_t$ is connected, so the connected component of $N_t$ to which $C_k$ belongs is simply $N_t$, and it  intersects $E_k$ at two points (at least), namely $C_k \cap E_k$ and $E_{k-1}\cap E_k$. 
The first point is equal to $p$ so $C_k$ is contained in the connected component of $N_{t}$ that passes through $p$ and by statement \ref{item4}  of Theorem \ref{pro:KLM} the function $\alpha(t)$ has a point of non-differentiability at $t=t_0$, i.e. there is an internal lower vertex at $t=t_0$.
The second intersection point $E_{k-1}\cap E_k$ of $N_{t}$ with $E_k$ is different from $p$, so the connected component containing $C_k$  intersects $E_k$ away from $p$ and by statement \ref{item5}  of Theorem \ref{pro:KLM} the function $\beta(t)$ has a point of non-differentiability at $t=t_0$ and there is an internal upper vertex at $t=t_0$. 
Thus $\Delta_{\Y}(D)$ has two internal vertices with abscissa $t_{C_k}$.
\end{proof}


\begin{theorem}\label{thm:characterize-picard-1}
	For every ample divisor $D$ on a normal projective surface $S$ there exist rank 2 valuations $v$ such that $\Delta_{v}(D)$ has at least four vertices. If $\rho(S)>1$, then there exist rank 2 valuations $v$ such that $\Delta_{v}(D)$ has more than four vertices.
\end{theorem}

\begin{corollary}
	Among all projective smooth surfaces $S$, those with Picard number 1 are characterized by the fact that for every big divisor $D$ and every rank 2 valuation $v$ of $K(S)$, the Newton--Okounkov polygon $\Delta_{v}(D)$ has at most 4 vertices.
\end{corollary}

For the proof of Theorem \ref{thm:characterize-picard-1} we need the following lemma.

\begin{lemma}\label{lem:no-additional-curve-large-k}
	Let $D$ be an ample divisor and $C$ a nodal curve on the surface $S$. 
	Let $k$ be an integer such that $k>(D\cdot C)/\sqrt{D^2}-1$ and $D-\frac{D\cdot C}{k+1}C$ is ample.
	Consider the construction above of surfaces $S_k$ with infinitesimal flags $\Y=\{S_k\supset E_k\supset \{p_k\}\}$, and suppose that $\Delta_{\Y}(D)$ has exactly four vertices.
	Then for every $t$ such that $D_t=D-tE_k$ is pseudoeffective, the negative part $N_t$ in the Zariski decomposition $D_t=P_t+N_t$ is supported on $\mathscr{N}=\{E_{1,k},\dots,E_{k-1,k}, C_k\}$.
\end{lemma}
\begin{proof}
	We argue by contradiction. Assume that there is an irreducible component $Z_k$ of $N_\mu$ not in $\mathscr{N}$.
	In particular, since $D$ is ample, $Z_k$ is the strict transform in $S_k$ of a curve $Z$ in $S$. Let us denote by $Z_i$ the strict transform of $Z$ in the sequence of blowups \(S_i \rightarrow S_{i-1} \rightarrow \dots \rightarrow S_1 \rightarrow S_0=S\). 
	
	By the results above, the four vertices of $\Delta_{\Y}(D)$ must be one leftmost vertex, one rightmost vertex, and two vertices with abscissa $t_C=\inf\{t : C_k \text{ is a component of } N_t\}$. 
	Since $D_0\cdot Z_k=D\cdot Z>0$ because $D$ is ample, $Z_k$ does not belong to $N_t$ for $t$ near zero. Therefore $\inf\{t : Z_k \text{ is a component of } N_t\}=t_C$ (otherwise $Z_k$ would create some additional vertex on $\Delta_{\Y}(D)$ with abscissa different from $0, t_C. \mu$), and in particular $P_{t_C}\cdot Z_k=P_{t_C}\cdot C_k=0$.
	Moreover, for every $t<t_C$, the Zariski decomposition $D_t=P_t+N_t$ agrees with the Zariski decomposition of $D_t$ relative to $\mathscr{N}\setminus\{C_k\}$ computed above, and in particular
	\begin{equation*}
	P_{t}=D-\frac{t}{k}E\qquad \mbox{for all} \ 0\le t \le t_C\ .
	\end{equation*}
	The equality $P_{t_C}\cdot C_k=0$ now allows us to determine $t_C=\frac{kD\cdot C}{k+1}$, which in turn implies a strong restriction on $Z_k$: let us denote by $m_{Z,i}$ the multiplicity at $p_i$ of the strict transform $Z_i$ of $Z$ in $S_i$.
	Then we have
	\begin{equation}\label{eq:same-t}
	0=P_{t_C}\cdot Z_k=\left(D-\frac{\frac{kD\cdot C}{k+1}}{k}E\right)\cdot\left(Z-\sum_i m_{Z,i}E_i\right)
	= D\cdot Z -\frac{D\cdot C}{k+1}\sum_i m_{Z,i},
	\end{equation}
	and on the other hand $0\le C_k\cdot C'=C_k\cdot Z_k=C\cdot Z-2 m_{Z,1}-\sum_{i=2}^k m_{Z,i} $; therefore we deduce that $\sum_{i=1}^k m_{Z,i} \le C\cdot Z$, and equation \eqref{eq:same-t} yields
	\begin{equation*}
	0\ge D\cdot Z-\frac{D\cdot C}{k+1}C\cdot Z=\left(D-\frac{D\cdot C}{k+1}C\right)\cdot Z.
	\end{equation*}
	But this is a contradiction, because $D-\frac{D\cdot C}{k+1}C$ is ample.	
\end{proof}
\begin{proof}[Proof of Theorem \ref{thm:characterize-picard-1}]
	The Newton--Okounkov bodies $\Delta_{\Y}(D)$ constructed in Proposition \ref{pro:nodal-curve-two-internal} have at least two internal vertices, a rightmost and a leftmost vertex, so they have at least four vertices.
	
	Now assume $\rho(S)>1$, and choose a nodal curve $C$ with $C^2>0$ whose numerical equivalence class is not a rational multiple of $[D]$.	
	We will show that in the construction above of an infinitesimal flag $\Y=\{S_k\supset E_k\supset \{p_k\}\}$, it is possible to choose $k$ such that the Newton--Okounkov body $\Delta_{\Y}(D)$ has five or more vertices. 

	We argue by contradiction. Assume that for every $k\ge (D\cdot C)/\sqrt{D^2}-1$, the polygon $\Delta_{\Y}(D)$ has exactly four vertices, namely one leftmost vertex, one rightmost vertex, and two vertices with abscissa $t_C=\inf\{t : C_k \text{ is a component of } N_t\}$. 
	By Lemma \ref{lem:no-additional-curve-large-k} this implies that, if $k$ is large enough, 
	then for all $t$ such that $D_t=D-tE_k$ is pseudoeffective, the negative part $N_t$ in the Zariski decomposition $D_t=P_t+N_t$ is supported on $\mathscr{N}=\{E_{1,k},\dots,E_{k-1,k}, C_k\}$, and hence $P_t$ can be computed as $P_{t,\mathscr{N}}$. 
	This is a straightforward (if somewhat cumbersome) linear algebra computation leading to
	\begin{equation}\label{eq:pmu}
	P_\mu=D-\frac{(k+1) \mu-k (D\cdot C)}{k^2-k(C^2-1)}C-\frac{\mu C^2-k(D\cdot C)}{k^2-k(C^2-1)}E.
	\end{equation}
	The assumption that there is a unique rightmost vertex implies that $P_\mu \cdot E_k=0$. The facts that $D\cdot E_k=C\cdot E_k=0$, and $E\cdot E_k=-1$, setting $P_\mu\cdot E_k=0$ imply $\mu=k(D\cdot C)/C^2$. 
	Substituting this value in \eqref{eq:pmu} and taking into account the intersection numbers between $D, C$ and $E$, one gets
	\[P_\mu^2=\frac{1}{C^2}\det\begin{pmatrix}D^2 & D\cdot C \\ D\cdot C & C^2 \end{pmatrix},\] 
	which is negative by the Index Theorem. 
	This is a contradiction, because $P_\mu$ is nef by definition.
\end{proof}

\begin{conjecture}
	Let $S$ be a smooth projective algebraic surface, and $D$ a big divisor on $S$. 
	There is a smooth projective surface $\tilde{S}$ with a proper projective morphism $\pi:\tilde{S}\rightarrow S$ and a flag $\Y=\{\tilde{S}\supset E \supset \{p\}\}$ such that $\Delta_{\Y}(D)$ has exactly $2\rho +2$ vertices.
\end{conjecture}

Note that the divisorial part $E$ of the flag belongs necessarily to $\Null(D)$. 
The use of a different model $\tilde{S}$ provides ``new'' curves on $\Null(D)$ to use in the flag, as was done above to prove Theorem \ref{thm:characterize-picard-1}, but it is a challenge to construct appropriate blowups $\pi:\tilde{S}\rightarrow S$ where one can guarantee the maximal number of internal vertices.

\section*{Acknowledgments}

We are grateful to Alex Küronya and Carlos Jes\'us Moreno-\'Avila for several discussions on the subject.
The first and second authors wish to thank the UAB (Universitat Aut\`onoma de Barcelona), where the first ideas of this paper came out, for the invitation and hospitality.

{\footnotesize
	\bibliographystyle{plainurl}
\bibliography{NOB}}

\end{document}